\documentclass{amsart}

\makeatletter
\@namedef{subjclassname@2020}{%
  \textup{2020} Mathematics Subject Classification}
\makeatother
\newtheorem{theorem}{Theorem}[section]

\theoremstyle{definition}

\newtheorem{corollary}[theorem]{Corollary}

\theoremstyle{remark}

\numberwithin{equation}{section}

\begin{document}

\title[ Bounds to the first eigenvalues of Steklov problems ]{Bounds to the first eigenvalues of weighted p-Steklov and (p,q)-Laplacian Steklov problems }

\author{Shahroud Azami}
\address{Department of  Pure Mathematics, Faculty of Science,
Imam Khomeini International University,
Qazvin, Iran. \\
              Tel.: +98-28-33901321\\
              Fax: +98-28-33780083\\
             }

\email{azami@sci.ikiu.ac.ir}
%
%




\begin{abstract}
We consider the Steklov problem  associated with the weighted p-Laplace operator and $(p,q)$-Laplacian on submanifolds with boundary of Euclidean spaces and prove Reilly-type upper bounds for their first eigenvalues.
\end{abstract}

\maketitle

\section{Introduction}
Let $(M^{n},g)$ be a compact Riemannian manifold with a possibility non-empty boundary $\partial M$. The triple $(M,g,d\mu_{g}=e^{-f}dv)$ is called a smooth metric measure space, where $f:M\to \mathbb{R}$ is a smooth real-valued function on $M$ and $dv$ is the Riemannian volume element related to $g$. We also call $e^{-f}$ the density.\\
For $1<p<\infty$ and  any $u\in W_{0}^{1,p}(M)$, the p-Laplacian $\Delta_{p}$ is defined by
\begin{equation*}
\Delta_{p}u={\rm div}(|\nabla u|^{p-2}\nabla u)=|\nabla u|^{p-2}\Delta u+(p-2)|\nabla u|^{p-4}{\rm Hess} u(\nabla u, \nabla u),
\end{equation*}
where ${\rm div}$ is the divergence operator, $\nabla $ is the gradient operator, and ${\rm Hess} u $ is the hessian of $u$. For $p=2$, the p-Laplacian  is the Laplace-Beltrami operator of $(M^{n},g)$. Also, the weighted p-Laplacian  is defined by
\begin{equation*}
\Delta_{p,f}u=e^{f}{\rm div}(e^{-f}|\nabla u|^{p-2}\nabla u).
\end{equation*}
The spectrum of the weighted p-Laplacian  has been studied on smooth metric measure spaces with Dirichlet or Neuman boundary conditions (see \cite{FED, HS, LW, YW}).
In the present paper, we will consider the Steklov problem associated  with the  weighted p-Laplace operator and $(p,q)$-Laplacian on submanifolds with boundary of Euclidean spaces.\\

In  following  we will consider the weighted $p$-Steklov problem on submanifolds with boundary of the Euclidean space
\begin{equation} \label{e1}
\begin{cases}
\Delta_{p,f}u=0&\text{in}  \, M\\|\nabla u|^{p-2}\frac{\partial u}{\partial \nu}=\lambda |u|^{p-2}u&\text {on}\,\partial M,
\end{cases}
\end{equation}
where $\frac{\partial u}{\partial \nu}$ is the derivative of the function $u$ with respect to the outward unit normal $\nu$ to the boundary $\partial M$.  If $f$ be a constant function then the weighted  p-Steklov problem (\ref{e1}) reduces to the p-Steklov problem which it has been studied in \cite{Ro4}. This problem arises from  the following variational characterization of the first positive eigenvalue given  by
\begin{equation}\label{e2}
\lambda_{1}(M)=\inf\left\{ \frac{\int_{M}|\nabla u|^{p}d\mu_{g}}{\int_{\partial M}|u|^{p}d\mu_{h}}\Big|u\in W^{1,p}(M)\setminus\{0\},\,\int_{\partial M}|u|^{p-2}ud\mu_{h}=0\right\}
\end{equation}
where $d\mu_{h}$ is the weighted measure on $\partial M$. Also, we consider  a Steklov problem associated with the $(p,q)$-Laplacian  as follow
\begin{equation}\label{pq}
\begin{cases}
\Delta_{p}u+\Delta_{q}u=0 &\text{in}\,\,\, M,\\
\left( |\nabla u|^{p-2}+|\nabla u|^{q-2}\right)\frac{\partial u}{\partial \nu}=\sigma|u|^{r-2}u &\text{on}\,\,\, \partial M,
\end{cases}
\end{equation}
where $1<p<q<r<\infty$, $r\in(1,\frac{q(N-1)}{N-q})$ if $p<N$ and $r\in(1,\infty)$ if $p\geq N$. The first positive eigenvalue of the  $(p,q)$-Steklov problem (\ref{pq}) defined as
\begin{equation}\label{pq1}
\sigma_{1}(M)=\inf\left\{\frac{\int_{M}\left(|\nabla u|^{p}+|\nabla u|^{q} \right)dv_{g}}{\int_{\partial M}|u|^{r}\,dv_{h}}\Big|u\in W^{1,q}(M)\setminus\{0\},\,\,\int_{\partial M}|u|^{r-2}udv_{h}=0 \right\}.
\end{equation}
The  $(p,q)$-Steklov problem has been studied in \cite{LB, LB1, LB2}.\\

The aim of this paper is to obtain upper bounds for the first positive eigenvalue  of the problems (\ref{e1}) and (\ref{pq}), for submanifolds of Riemannian manifolds, depending on the geometry of boundary in the spirit of the classical Reilly upper bounds for the Laplacian on closed hypersurfaces.

For the first positive eigenvalue $\lambda_{1}$ of Laplacian, Reilly \cite{RI} proved the following well-known upper bound
\begin{equation*}
\lambda_{1}\leq \frac{n}{Vol(M)}\int_{M}H^{2}dv_{g},
\end{equation*}
where $H$ is the mean curvature of the immersion. Also, he \cite{RI} showed that for $r\in\{1,2,\cdots, n\}$,
\begin{equation*}
\lambda_{1}\left(\int_{M}H_{r-1}dv_{g}\right)^{2}\leq Vol(M)\int_{M}H_{r}^{2}dv_{g},
\end{equation*}
 where  $H_{r}$  is the $r$-th mean curvature of the immersion  and  defined by the $r$-th symmetric polynomial of the principal curvatures. Moreover, Reilly studied the equality  cases and proved that  equality holds in one these inequalities, if  and only if  $M$ is immersed in a geodesic sphere of radius $\sqrt{\frac{n}{\lambda_{1}}}$. More generally, he show that if $(M^{n},g)$ is isometrically immersed into $\mathbb{R}^{N}$, $N>n+1$, then
\begin{equation*}
\lambda_{1}\left(\int_{M}H_{r}dv_{g}\right)^{2}\leq Vol(M)\int_{M}|{ H_{r+1}}|^{2}dv_{g},
\end{equation*}
for any even $r\in\{0,1,\cdots,n\}$ and equality holds if and only if $M$ is minimally immersed in a geodesic sphere of $\mathbb{R}^{N}$. For codimension greater than 1, $H_{r}$ is  a function and ${ H}_{r+1}$ is a normal vector field.
 These inequalities have been  generalized for other ambient spaces and other operators (see \cite{AH,  AL, BA, Bat, Do,  El, Ro1, Ro2, Ro3, Ro4}).\\
Du and Mao \cite{DMA} established the first positive eigenvalue of the $p$-Laplacian on closed submanifold of $\mathbb{R}^{N}$ satisfies the follows inequalities.
\begin{equation*}
\lambda_{1}\leq \frac{n^{\frac{p}{2}}}{(Vol(M))^{p}}\left(\int_{M}|H|^{\frac{p}{p-1}} dv_{g}\right)^{p-1}\begin{cases}
N^{\frac{2-p}{2}} &\text{if}\,\,\,1<p\leq2,\\N^{\frac{p-2}{2}} &\text{if}\,\,\,p\geq2.
\end{cases}
\end{equation*}
In addition, equality holds if and only if $p=2$ and $M$ is minimally immersed into  a geodesic hypersphere.  On the other hand, Roth \cite{Ro4} proved Reilly-type inequalities  for the first eigenvalue of $p$-Steklov problem on submanifolds of $\mathbb{R}^{N}$ and showed that
\begin{equation*}
\lambda_{1}\leq \frac{Vol(M)}{(Vol(\partial M))^{p}}n^{\frac{p}{2}}\left(\int_{M}|H|^{\frac{p}{p-1}} \right)^{p-1}\begin{cases}
N^{\frac{2-p}{2}} &\text{if}\,\,\,1<p\leq2,\\N^{\frac{p-2}{2}} &\text{if}\,\,\,p\geq2.
\end{cases}
\end{equation*}
Moreover, equality  holds if and only if $p=2$ and $M$ is minimally immersed into   $B^{N}(\frac{1}{\lambda_{1}})$ such that $X(\partial M)\subset \partial B^{N}(\frac{1}{\lambda_{1}})$, where $X$ is the isometric immersion.
\section{Main results}
Motivated by above works, we prove that:
\begin{theorem}\label{t1}
Let $(M^{n}, g, d\mu=e^{-f}dv)$ be a compact connected  and oriented Riemannian manifold with nonempty boundary $\partial M$ and $p\in(1,+\infty)$. Assume that $(M^{n}, g, d\mu=e^{-f}dv)$ isometrically immersied into the Euclidean space $\mathbb{R}^{N}$ by $X$. If $\lambda_{1}(M)$ is the eigenvalue of the  weighted $p$-Steklov problem (\ref{e1}) then  for $1<p\leq2$ we have
\begin{equation*}
\lambda_{1}(M)\leq 2^{\frac{1}{p-1}} n^{-\frac{p}{2}}N^{1-\frac{p}{2}}
\left( \int_{\partial M}\left(|nH|^{\frac{p}{p-1}}+|\nabla f|^{\frac{p}{p-1}}\right)d\mu_{h}\right)^{p-1}
\frac{Vol_{\mu_{g}}(M)}{(Vol_{\mu_{h}}(\partial M))^{p}}
\end{equation*}
and for $p\geq2$ we get
\begin{equation*}
\lambda_{1}(M)\leq 2^{\frac{1}{p-1}} n^{-\frac{p}{2}}N^{\frac{p}{2}-1}
\left( \int_{\partial M}\left(|nH|^{\frac{p}{p-1}}+|\nabla f|^{\frac{p}{p-1}}\right)d\mu_{h}\right)^{p-1}
\frac{Vol_{\mu_{g}}(M)}{(Vol_{\mu_{h}}(\partial M))^{p}}.
\end{equation*}
where $Vol_{\mu_{g}}(\partial M)=\int_{ M}d\mu_{g}$ and   $Vol_{\mu_{h}}(\partial M)=\int_{\partial M}d\mu_{h}$. Moreover,
\begin{itemize}
\item[(i)] If $f$ is constant, $H$ does not vanish identically then equality occurs in both inequality if and only if  $p=2$ and $M$ is minimally immersed into $B^{N}(\frac{1}{\lambda_{1}(M)})$  so that  $\partial M$ lies  into  geodesic hypersphere  $\partial B^{N}(\frac{1}{\lambda_{1}(M)})$.
\item[(ii)] If  $f$ is not constant and  if equality occurs   then $M$ is a self-shrinker for the mean curvature flow and $f_{|_{M}}=a-\frac{b}{2}r_{p}^{2}$ for some constants $a,b$, where  $r_{p}$  is the Euclidean distance to the center of mass $p$ of $M$. In particular, if $n=N-1$ and $H>0$ or $n=2$, $N=3$ and $M$ is embedded  and has genus $0$, then $M$ is a geodesic ball.
\end{itemize}
\end{theorem}
  Let $T$ be a symmetric positive definite and divergence-free $(1,1)$-tensor on $M$. We  associated with $T$ the normal vecor field $H_{T}$ defined by
\begin{equation*}
H_{T}=\sum_{i,j=1}^{n}\langle Te_{i},e_{j}\rangle  B(e_{i},e_{j}),
\end{equation*}
where $\{e_{1},\cdots, e_{n}\}$ is a local orthonormal  frame  of $T\partial M$ and $B$ is the second fundamental form of the immersion of  $M$ into $ \mathbb{R}^{N}$. We also, recall the generalized Hsiung-Minkowski formula \cite{El, Ro2, Ro3} as
\begin{equation}\label{eq1}
\int_{\partial M}\left( \langle X, H_{T}-T(\nabla f)\rangle +{\rm tr}(T)\right)d\mu_{h}=0.
\end{equation}
In following theorem we extended the theorem  \ref{t1}   to estimates  with higher order mean curvatures.
\begin{theorem}\label{t2}
Let $(M^{n}, g, d\mu=e^{-f}dv)$ be a compact connected  and oriented Riemannian manifold with nonempty boundary $\partial M$ and $p\in(1,+\infty)$. Assume that $(M^{n}, g, d\mu=e^{-f}dv)$ isometrically immersed into the Euclidean space $\mathbb{R}^{N}$ by $X$ and let $T$ be a symmetric and divergence-free $(0,2)$-tensor on $\partial M$. If $\lambda_{1}(M)$ is the eigenvalue of the  weighted $p$-Steklov problem (\ref{e1}) then  for $1<p\leq2$ we have
\begin{eqnarray*}
&&\lambda_{1}(M)\left|\int_{\partial M}{\rm tr}(T) d\mu_{h}  \right|^{p}\\&&\leq 2^{\frac{1}{p-1}} n^{\frac{p}{2}}N^{1-\frac{p}{2}}
\left( \int_{\partial M}\left(|H_{T}|^{\frac{p}{p-1}}+|\nabla f|^{\frac{p}{p-1}}\right)d\mu_{h}\right)^{p-1}
Vol_{\mu_{g}}(M)
\end{eqnarray*}
and for $p\geq2$ we get
\begin{eqnarray*}
&&\lambda_{1}(M)\left|\int_{\partial M}{\rm tr}(T)d\mu_{h} \right|^{p}\\&&\leq 2^{\frac{1}{p-1}} n^{\frac{p}{2}}N^{\frac{p}{2}-1}
\left( \int_{\partial M}\left(|H_{T}|^{\frac{p}{p-1}}+|\nabla f|^{\frac{p}{p-1}}\right)d\mu_{h}\right)^{p-1}
Vol_{\mu_{g}}(M).
\end{eqnarray*}
where $Vol_{\mu_{g}}(\partial M)=\int_{ M}d\mu_{g}$ and   $Vol_{\mu_{h}}(\partial M)=\int_{\partial M}d\mu_{h}$. Moreover,
\begin{itemize}
\item[(i)] If $f$ is constant, $H_{T}$ does not vanish identically then equality occurs in both inequality if and only if  $p=2$ and $M$ is minimally immersed into $B^{N}(\frac{1}{\lambda_{1}(M)})$  so that  $\partial M$ lies  into  geodesic hypersphere  $\partial B^{N}(\frac{1}{\lambda_{1}(M)})$.
\item[(ii)] If  $f$ is not constant and  if equality occurs   then $M$ is a self-shrinker for the mean curvature flow and $f_{|_{M}}=a-\frac{b}{2}r_{p}^{2}$ for some constants $a,b$, where  $r_{p}$  is the Euclidean distance to the center of mass $p$ of $M$. In particular, if $n=N-1$ and $H>0$ or $n=2$, $N=3$ and $M$ is embedded  and has genus $0$, then $M$ is a geodesic ball.
\end{itemize}
\end{theorem}

For $r\in\{1,\cdots,n\}$, let
\begin{equation*}
T_{r}=\frac{1}{r!}\sum_{\begin{subarray}{l}
  i,{i_1},...,{i_r} \\
  j,{j_1},...,{j_r} \\
\end{subarray} } {\varepsilon \left( \begin{gathered}
  i,{i_1},...,{i_r} \hfill \\
  j,{j_1},...,{j_r} \hfill \\
\end{gathered}  \right)} \langle B_{i_{1}j_{1}}, B_{i_{2}j_{2}}\rangle\cdots \langle B_{i_{r-1}j_{r-1}}, B_{i_{r}j_{r}}\rangle e_{i}^{*}\otimes e_{j}^{*}
\end{equation*}
if $r$ is even and
\begin{equation*}
T_{r}=\frac{1}{r!}\sum_{\begin{subarray}{l}
  i,{i_1},...,{i_r} \\
  j,{j_1},...,{j_r} \\
\end{subarray} } {\varepsilon \left( \begin{gathered}
  i,{i_1},...,{i_r} \hfill \\
  j,{j_1},...,{j_r} \hfill \\
\end{gathered}  \right)} \langle B_{i_{1}j_{1}}, B_{i_{2}j_{2}}\rangle\cdots \langle B_{i_{r-2}j_{r-2}}, B_{i_{r-1}j_{r-1}}\rangle B_{i_{r}j_{r}}\otimes  e_{i}^{*}\otimes e_{j}^{*}
\end{equation*}
 if $r$ is odd, where the $B_{ij}$'s  are the coefficients  of the second fundamental form  $B$ in  a orthonormal frame $\{e_{1},\cdots, e_{n}\}$ with the dual coframe $\{e_{1}^{*},\cdots, e_{n}^{*}\}$
and $\epsilon$ is the standard signature for permutations. The $r$-th mean curvature  is defined as
 $H_{0}=0$ and $H_{r}=\frac{1}{(n-r){r\choose n}}{\rm tr}(T_{r})$. If $r$ is even then $H_{r}$ is a real function and if $r$ is odd then $H_{r}$ is a normal vector field, in this case, we will denote it ${\bf H}_{r}$.  Also, the Hsiung-Minkowski formula becomes
\begin{equation*}
\int_{\partial M}\left( \langle X, {\bf H}_{r+1}\rangle+H_{r}\right)d\mu_{h}=0
\end{equation*}
for any even $r\in\{0,1,\cdots, n\}$ if $N>n+1$, and
\begin{equation*}
\int_{\partial M}\left( \langle X,\nu \rangle H_{r+1}+H_{r}\right)d\mu_{h}=0
\end{equation*}
for any  $r\in\{0,1,\cdots, n\}$ if $N=n+1$, where $\nu$ is the normal unit vector field on $\partial M$ chosen to define the shape operator.
\\
Now we obtain the following corollary from Theorem  \ref{t2}.
\begin{corollary}\label{c2}
Let $(M^{n}, g, d\mu=e^{-f}dv)$ be a compact connected  and oriented Riemannian manifold with nonempty boundary $\partial M$ and $p\in(1,+\infty)$. Assume that $(M^{n}, g, d\mu=e^{-f}dv)$ isometrically immersed into the Euclidean space $\mathbb{R}^{N}$ by $X$. If $\lambda_{1}(M)$ is the eigenvalue of the  weighted $p$-Steklov problem (\ref{e1})
\begin{itemize}
\item[(1)] If $N>n+1$, and $r\in\{0,\cdots,n-1\}$ is an even integer then we have
\begin{itemize}
\item[(a)] If  $1<p\leq2$ we have
\begin{eqnarray*}
&&\lambda_{1}(M)\left|\int_{\partial M}H_{r} d\mu_{h}  \right|^{p}\\&&\leq 2^{\frac{1}{p-1}} n^{\frac{p}{2}}N^{1-\frac{p}{2}}
\left( \int_{\partial M}\left(|{\bf H}_{r+1}|^{\frac{p}{p-1}}+|\nabla f|^{\frac{p}{p-1}}\right)d\mu_{h}\right)^{p-1}
Vol_{\mu_{g}}(M)
\end{eqnarray*}
\item[(b)]
If  $p\geq2$ we have
\begin{eqnarray*}
&&\lambda_{1}(M)\left|\int_{\partial M}H_{r} d\mu_{h}  \right|^{p}\\&&\leq 2^{\frac{1}{p-1}} n^{\frac{p}{2}}N^{\frac{p}{2}-1}
\left( \int_{\partial M}\left(|{\bf H}_{r+1}|^{\frac{p}{p-1}}+|\nabla f|^{\frac{p}{p-1}}\right)d\mu_{h}\right)^{p-1}
Vol_{\mu_{g}}(M)
\end{eqnarray*}
\end{itemize}
Moreover, if $f$ is constant, $H_{r}$ does not vanish identically then equality occurs in both inequality if and only if  $p=2$ and $M$ is minimally immersed into $B^{N}(\frac{1}{\lambda_{1}(M)})$  so that  $\partial M$ lies  into  geodesic hypersphere  $\partial B^{N}(\frac{1}{\lambda_{1}(M)})$.
\item[(2)] If $N=n+1$, and $r\in\{0,\cdots,n-1\}$ is an even integer then we have
\begin{itemize}
\item[(a)] If  $1<p\leq2$ we have
\begin{eqnarray*}
&&\lambda_{1}(M)\left|\int_{\partial M}H_{r} d\mu_{h}  \right|^{p}\\&&\leq 2^{\frac{1}{p-1}} n^{\frac{p}{2}}N^{1-\frac{p}{2}}
\left( \int_{\partial M}\left(|{ H}_{r+1}|^{\frac{p}{p-1}}+|\nabla f|^{\frac{p}{p-1}}\right)d\mu_{h}\right)^{p-1}
Vol_{\mu_{g}}(M)
\end{eqnarray*}
\item[(b)]
If  $p\geq2$ we have
\begin{eqnarray*}
&&\lambda_{1}(M)\left|\int_{\partial M}H_{r} d\mu_{h}  \right|^{p}\\&&\leq 2^{\frac{1}{p-1}} n^{\frac{p}{2}}N^{\frac{p}{2}-1}
\left( \int_{\partial M}\left(|{ H}_{r+1}|^{\frac{p}{p-1}}+|\nabla f|^{\frac{p}{p-1}}\right)d\mu_{h}\right)^{p-1}
Vol_{\mu_{g}}(M)
\end{eqnarray*}
\end{itemize}
Moreover, if $f$ is constant, $H_{r+1}$ does not vanish identically then equality occurs in both inequality if and only if  $p=2$ and  $X(M)=B^{N}(\frac{1}{\lambda_{1}(M)})$.
\end{itemize}
\end{corollary}
In following we investigate the first eigenvalue of weighted $p$-Steklov problem on Riemannian products $\mathbb{R}\times M$ where $M$ is  a complete Riemannian manifold.
\begin{theorem}\label{t3}
 Let $p\geq2$ and $(M^{n},\bar{g})$ be  a complete Riemannian manifold. Consider $(\Sigma^{n},g)$ a closed oriented Riemannian manifold isometrically  immersed into the Riemannian product $(\mathbb{R}\times M, \tilde{g}=dt^{2}\oplus \bar{g})$ with a density $e^{-f}$.  Moreover, assume that $\Sigma$ is mean-convex and bounds a domain $\Omega$ in $\mathbb{R}\times M$. Let $\lambda_{1}(M)$ be the first eigenvalue of the weighted $p$-Steklov problem on $\Omega$, then
\begin{equation*}
\lambda_{1}(\Omega)\leq
\left( \frac{\kappa_{+}(\Sigma)|H| _{\infty}}{\mathop {\inf } \limits_{\Sigma} H}\right)^{\frac{p}{2}}
\left(\frac{Vol_{\mu_{g}}(\Omega)}{Vol_{\mu_{h}}(\Sigma)} \right)^{1-\frac{p}{2}},
\end{equation*}
where $\kappa_{+}(\Sigma)=\max\{\kappa_{+}(x)|x\in M\}$ with $\kappa_{+}$ the biggest principal curvature of $\Sigma$ at the point $x$.
\end{theorem}
Now, we obtain Reilly upper bounds for $(p,q)$-Steklov problem. Similar the theorem \ref{t1} we have
\begin{theorem}\label{tpq1}
Let $(M^{n}, g, dv)$ be a compact connected  and oriented Riemannian manifold with nonempty boundary $\partial M$ and $1<p<q<r<\infty$. Assume that $(M^{n}, g, dv)$ isometrically immersed into the Euclidean space $\mathbb{R}^{N}$ by $X$. If $\sigma_{1}(M)$ is the eigenvalue of the   $(p,q)$-Steklov problem (\ref{pq}) then
\begin{itemize}
\item[(1)]
  for $1<p<q<r\leq2$ we have
\begin{equation*}
\sigma_{1}(M)\leq \left( N^{1-\frac{p}{2}}n^{\frac{p}{2}}+N^{1-\frac{q}{2}}n^{\frac{q}{2}}\right)
\left( \int_{\partial M}\left(|H|^{\frac{r}{r-1}}\right)dv_{h}\right)^{r-1}
\frac{Vol(M)}{(Vol(\partial M))^{r}}.
\end{equation*}
\item[(2)] For $1<p<q\leq2$ and $r>2$ we get
\begin{equation*}
\sigma_{1}(M)\leq N^{\frac{r}{2}-1}
\left( N^{1-\frac{p}{2}}n^{\frac{p}{2}}+N^{1-\frac{q}{2}}n^{\frac{q}{2}}\right)
\left( \int_{\partial M}\left(|H|^{\frac{r}{r-1}}\right)dv_{h}\right)^{r-1}
\frac{Vol(M)}{(Vol(\partial M))^{r}}.
\end{equation*}
\item[(3)] For $1<p\leq2$ and $2< q<r$ we get
\begin{equation*}
\sigma_{1}(M)\leq N^{\frac{r}{2}-1}\left( N^{1-\frac{p}{2}}n^{\frac{p}{2}}+n^{\frac{q}{2}}\right)
\left( \int_{\partial M}\left(|H|^{\frac{r}{r-1}}\right)dv_{h}\right)^{r-1}
\frac{Vol(M)}{(Vol(\partial M))^{r}}.
\end{equation*}
\item[(4)] For $2\leq p<q<r$ we get
\begin{equation*}
\sigma_{1}(M)\leq N^{\frac{r}{2}-1} \left( n^{\frac{p}{2}}+n^{\frac{q}{2}}\right)
\left( \int_{\partial M}\left(|H|^{\frac{r}{r-1}}\right)dv_{h}\right)^{r-1}
\frac{Vol(M)}{(Vol(\partial M))^{r}}.
\end{equation*}
\end{itemize}
\end{theorem}
In following theorem we extended the theorem  \ref{tpq1}   to estimates  with higher order mean curvatures.

\begin{theorem}\label{tpq2}
Let $(M^{n}, g, dv)$ be a compact connected  and oriented Riemannian manifold with nonempty boundary $\partial M$ and $1<p<q<r$. Assume that $(M^{n}, g, dv)$ isometrically immersed into the Euclidean space $\mathbb{R}^{N}$ by $X$ and let $T$ be a symmetric and divergence-free $(0,2)$-tensor on $\partial M$. If $\sigma_{1}(M)$ is the eigenvalue of the   $(p,q)$-Steklov problem (\ref{pq}) then
\begin{itemize}
\item[(1)]
  for $1<p<q<r\leq2$ we have
\begin{eqnarray*}
&&\sigma_{1}(M)\left|\int_{\partial M}{\rm tr}(T) dv_{h}  \right|^{r}\\&&\leq \left( N^{1-\frac{p}{2}}n^{\frac{p}{2}}+N^{1-\frac{q}{2}}n^{\frac{q}{2}}\right)
\left( \int_{\partial M}\left(|H_{T}|^{\frac{r}{r-1}}\right)dv_{h}\right)^{r-1}
Vol(M).
\end{eqnarray*}
\item[(2)] For $1<p<q\leq2$ and $r>2$ we get
\begin{eqnarray*}
&&\sigma_{1}(M)\left|\int_{\partial M}{\rm tr}(T) dv_{h}  \right|^{r}\\&&\leq N^{\frac{r}{2}-1}
\left( N^{1-\frac{p}{2}}n^{\frac{p}{2}}+N^{1-\frac{q}{2}}n^{\frac{q}{2}}\right)
\left( \int_{\partial M}\left(|H_{T}|^{\frac{r}{r-1}}\right)dv_{h}\right)^{r-1}
Vol(M).
\end{eqnarray*}
\item[(3)] For $1<p\leq2$ and $2< q<r$ we get
\begin{eqnarray*}
&&\sigma_{1}(M)\left|\int_{\partial M}{\rm tr}(T) dv_{h}  \right|^{r}\\&&\leq N^{\frac{r}{2}-1}\left( N^{1-\frac{p}{2}}n^{\frac{p}{2}}+n^{\frac{q}{2}}\right)
\left( \int_{\partial M}\left(|H_{T}|^{\frac{r}{r-1}}\right)dv_{h}\right)^{r-1}
Vol(M).
\end{eqnarray*}
\item[(4)] For $2\leq p<q<r$ we get
\begin{eqnarray*}
&&\sigma_{1}(M)\left|\int_{\partial M}{\rm tr}(T) dv_{h}  \right|^{r}\\&&\leq N^{\frac{r}{2}-1} \left( n^{\frac{p}{2}}+n^{\frac{q}{2}}\right)
\left( \int_{\partial M}\left(|H_{T}|^{\frac{r}{r-1}}\right)dv_{h}\right)^{r-1}
Vol(M).
\end{eqnarray*}
\end{itemize}
 \end{theorem}
Also, we have
\begin{corollary}\label{tpq3}
Let $(M^{n}, g, dv)$ be a compact connected  and oriented Riemannian manifold with nonempty boundary $\partial M$ and $1<p<q<r$. Assume that $(M^{n}, g, dv)$ isometrically immersed into the Euclidean space $\mathbb{R}^{N}$ by $X$. If $\sigma_{1}(M)$ is the eigenvalue of the   $(p,q)$-Steklov problem (\ref{pq})
\begin{itemize}
\item[(i)] If $N>n+1$, and $s\in\{0,\cdots,n-1\}$ is an even integer then we have
\begin{itemize}
\item[(1)]
  for $1<p<q<r\leq2$ we have
\begin{eqnarray*}
&&\sigma_{1}(M)\left|\int_{\partial M}H_{s} dv_{h}  \right|^{r}\\&&\leq \left( N^{1-\frac{p}{2}}n^{\frac{p}{2}}+N^{1-\frac{q}{2}}n^{\frac{q}{2}}\right)
\left( \int_{\partial M}\left(|{\bf H}_{s+1}|^{\frac{r}{r-1}}\right)dv_{h}\right)^{r-1}
Vol(M).
\end{eqnarray*}
\item[(2)] For $1<p<q\leq2$ and $r>2$ we get
\begin{eqnarray*}
&&\sigma_{1}(M)\left|\int_{\partial M}H_{s} dv_{h}  \right|^{r}\\&&\leq N^{\frac{r}{2}-1}
\left( N^{1-\frac{p}{2}}n^{\frac{p}{2}}+N^{1-\frac{q}{2}}n^{\frac{q}{2}}\right)
\left( \int_{\partial M}\left(|{\bf H}_{s+1}|^{\frac{r}{r-1}}\right)dv_{h}\right)^{r-1}
Vol(M).
\end{eqnarray*}
\item[(3)] For $1<p\leq2$ and $2< q<r$ we get
\begin{eqnarray*}
&&\sigma_{1}(M)\left|\int_{\partial M}H_{s} dv_{h}  \right|^{r}\\&&\leq N^{\frac{r}{2}-1}\left( N^{1-\frac{p}{2}}n^{\frac{p}{2}}+n^{\frac{q}{2}}\right)
\left( \int_{\partial M}\left(|{\bf H}_{s+1}|^{\frac{r}{r-1}}\right)dv_{h}\right)^{r-1}
Vol(M).
\end{eqnarray*}
\item[(4)] For $2\leq p<q<r$ we get
\begin{eqnarray*}
&&\sigma_{1}(M)\left|\int_{\partial M}H_{s} dv_{h}  \right|^{r}\\&&\leq N^{\frac{r}{2}-1} \left( n^{\frac{p}{2}}+n^{\frac{q}{2}}\right)
\left( \int_{\partial M}\left(|{\bf H}_{s+1}|^{\frac{r}{r-1}}\right)dv_{h}\right)^{r-1}
Vol(M).
\end{eqnarray*}
\end{itemize}

\item[(ii)] If $N=n+1$, and $s\in\{0,\cdots,n-1\}$ is an even integer then we have
\begin{itemize}
\item[(1)]
  for $1<p<q<r\leq2$ we have
\begin{eqnarray*}
&&\sigma_{1}(M)\left|\int_{\partial M}H_{s} dv_{h}  \right|^{r}\\&&\leq \left( N^{1-\frac{p}{2}}n^{\frac{p}{2}}+N^{1-\frac{q}{2}}n^{\frac{q}{2}}\right)
\left( \int_{\partial M}\left(|H_{s}|^{\frac{r}{r-1}}\right)dv_{h}\right)^{r-1}
Vol(M).
\end{eqnarray*}
\item[(2)] For $1<p<q\leq2$ and $r>2$ we get
\begin{eqnarray*}
&&\sigma_{1}(M)\left|\int_{\partial M}H_{s} dv_{h}  \right|^{r}\\&&\leq N^{\frac{r}{2}-1}
\left( N^{1-\frac{p}{2}}n^{\frac{p}{2}}+N^{1-\frac{q}{2}}n^{\frac{q}{2}}\right)
\left( \int_{\partial M}\left(|H_{s}|^{\frac{r}{r-1}}\right)dv_{h}\right)^{r-1}
Vol(M).
\end{eqnarray*}
\item[(3)] For $1<p\leq2$ and $2< q<r$ we get
\begin{eqnarray*}
&&\sigma_{1}(M)\left|\int_{\partial M}H_{s} dv_{h}  \right|^{r}\\&&\leq N^{\frac{r}{2}-1}\left( N^{1-\frac{p}{2}}n^{\frac{p}{2}}+n^{\frac{q}{2}}\right)
\left( \int_{\partial M}\left(|H_s|^{\frac{r}{r-1}}\right)dv_{h}\right)^{r-1}
Vol(M).
\end{eqnarray*}
\item[(4)] For $2\leq p<q<r$ we get
\begin{eqnarray*}
&&\sigma_{1}(M)\left|\int_{\partial M}H_{s} dv_{h}  \right|^{r}\\&&\leq N^{\frac{r}{2}-1} \left( n^{\frac{p}{2}}+n^{\frac{q}{2}}\right)
\left( \int_{\partial M}\left(|H_{s}|^{\frac{r}{r-1}}\right)dv_{h}\right)^{r-1}
Vol(M).
\end{eqnarray*}
\end{itemize}
\end{itemize}
\end{corollary}
In following we investigate the first eigenvalue of $(p,q)$-Steklov problem on Riemannian products $\mathbb{R}\times M$ where $M$ is  a complete Riemannian manifold.
\begin{theorem}\label{tpq4}
 Let $2\leq p<q<r$ and $(M^{n},\bar{g})$ be  a complete Riemannian manifold. Consider $(\Sigma^{n},g)$ a closed oriented Riemannian manifold isometrically  immersed into the Riemannian product $(\mathbb{R}\times M, \tilde{g}=dt^{2}\oplus \bar{g})$.  Moreover, assume that $\Sigma$ is mean-convex and bounds a domain $\Omega$ in $\mathbb{R}\times M$. Let $\sigma_{1}(M)$ be the first eigenvalue of the weighted $p$-Steklov problem on $\Omega$, then
\begin{equation*}
\sigma_{1}(\Omega)\leq 2
\left( \frac{\kappa_{+}(\Sigma)|H| _{\infty}}{\mathop {\inf } \limits_{\Sigma} H}\right)^{\frac{r}{2}}
\left(\frac{Vol(\Omega)}{Vol(\Sigma)} \right)^{1-\frac{r}{2}},
\end{equation*}
where $\kappa_{+}(\Sigma)=\max\{\kappa_{+}(x)|x\in M\}$ with $\kappa_{+}$ the biggest principal curvature of $\Sigma$ at the point $x$.
\end{theorem}

\section{ Proof of main results}
 In this section we give the proof of our main results.
\begin{proof}[proof of theorem \ref{t1}]
For coordinates functions $X^{k}$, by replacing if needed, $|X^{i}|^{p-2}X^{i}$ by $$|X^{i}|^{p-2}X^{i}-\frac{\int_{\partial M}|X^{i}|^{p-2}X^{i}d\mu_{h}}{Vol_{\mu_{h}}(\partial M)} $$
we can assume without loss  of generality,
\begin{equation*}
\int_{\partial M}|X^{i}|^{p-2}X^{i}d\mu_{h}=0
\end{equation*}
for all $i\in\{1,2,\cdots,N\}$. Thus, we can use the coordinates functions $X^{k}$ as test functions. \\

The case $1<p\leq2$.\\
By the definition of $\lambda_{1}(M)$ we have
\begin{equation}\label{1}
\lambda_{1}(M)\int_{\partial M}\sum_{i=1}^{N}|X^{i}|^{p}d\mu_{h}\leq\int_{M}\sum_{i=1}^{N}|\nabla X^{i}|^{p}d\mu_{g}.
\end{equation}
Since $p\leq 2$,  we get $\left(\sum_{i=1}^{N}|X^{i}|^{2}\right)^{\frac{1}{2}}\leq \left(\sum_{i=1}^{N}|X^{i}|^{p}\right)^{\frac{1}{p}}$, then
\begin{equation}\label{2}
|X|^{p}=\left(\sum_{i=1}^{N}|X^{i}|^{2}\right)^{\frac{p}{2}}
\leq\sum_{i=1}^{N}| X^{i}|^{p}.
\end{equation}
On the other hand, the concavity of $y\to y^{\frac{p}{2}}$ yields
\begin{equation}\label{3}
\sum_{i=1}^{N}|\nabla X^{i}|^{p}=\sum_{i=1}^{N}\left(| \nabla X^{i}|^{2}\right)^{\frac{p}{2}}\leq N^{1-\frac{p}{2}}\left(\sum_{i=1}^{N}|\nabla X^{i}|^{2}\right)^{\frac{p}{2}}=N^{1-\frac{p}{2}} n^{\frac{p}{2}},
\end{equation}
since we have $\sum_{i=1}^{N}|\nabla X^{i}|^{2}=n$ (see \cite[Lemma 2.1]{Ro2}).
Hence, we obtain
\begin{equation}\label{4}
\lambda_{1}(M)\int_{\partial M}|X|^{p}d\mu_{h}\leq N^{1-\frac{p}{2}}n^{\frac{p}{2}}Vol_{\mu_{g}}(M).
\end{equation}
On the other hand, using H\"{o}lder inequality we have
\begin{eqnarray*}
&&\int_{\partial M}\langle X, nH-\nabla f\rangle d\mu_{h}\\&&\leq \left( \int_
{\partial M}|X|^{p}d\mu_{h}\right)^{\frac{1}{p}}\left(  \int_
{\partial M}|n H-\nabla f|^{\frac{p}{p-1}}d\mu_{h}\right)^{\frac{p-1}{p}}\\&\leq&
\left( \int_
{\partial M}|X|^{p}d\mu_{h}\right)^{\frac{1}{p}}\left( 2^{\frac{1}{p-1}} \int_
{\partial M}(|n H|^{\frac{p}{p-1}}+|\nabla f|^{\frac{p}{p-1}})d\mu_{h}\right)^{\frac{p-1}{p}}
\end{eqnarray*}
With multiply  both sides of (\ref{4}) by $\left( \int_{\partial M}\left(|nH|^{\frac{p}{p-1}}+|\nabla f|^{\frac{p}{p-1}}\right)d\mu_{h}\right)^{p-1}$ and use the integral H\"{o}lder inequality, we conclude that
\begin{eqnarray}\label{5}
&&2^{-\frac{1}{p-1}}\lambda_{1}(M)\left|\int_{\partial M}\langle X, nH-\nabla f\rangle d\mu_{h}\right|^{p}\\\nonumber&&\leq N^{1-\frac{p}{2}}n^{\frac{p}{2}}
\left( \int_{\partial M}\left(|nH|^{\frac{p}{p-1}}+|\nabla f|^{\frac{p}{p-1}}\right)d\mu_{h}\right)^{p-1}
Vol_{\mu_{g}}(M).
\end{eqnarray}
Now,  using the Hsiung-Minkouski formula
\begin{equation}\label{6}
\int_{\partial M}\left(\langle X, nH-\nabla f\rangle+n \right)d\mu_{h}=0
\end{equation}
we infer
\begin{eqnarray}\label{7}
&&2^{-\frac{1}{p-1}}\lambda_{1}(M)\left(nVol_{\mu_{h}}(\partial M)\right)^{p}\\\nonumber&&\leq N^{1-\frac{p}{2}}n^{\frac{p}{2}}
\left( \int_{\partial M}\left(|nH|^{\frac{p}{p-1}}+|\nabla f|^{\frac{p}{p-1}}\right)d\mu_{h}\right)^{p-1}
Vol_{\mu_{g}}(M).
\end{eqnarray}
The inequality is proven. First, assume that  $f$ is constant, $H$ does not vanish identically, and   equality holds. Then the end of the proof is similar to the  proof of Roth \cite{Ro4} for the $p$-Steklov problem.  Now, assume that  $f$ is not constant. If equality occurs, then the end of the proof is similar to the proof of Roth \cite{Ro3}.\\

The case $p\geq2$.\\
It is straightforward that
\begin{equation}\label{8}
\sum_{i=1}^{N}|\nabla X^{i}|^{p}=\sum_{i=1}^{N}\left(| \nabla X^{i}|^{2}\right)^{\frac{p}{2}}\leq \left(\sum_{i=1}^{N}|\nabla X^{i}|^{2}\right)^{\frac{p}{2}}= n^{\frac{p}{2}}.
\end{equation}
On the other hand, using the fact that $y\to y^{\frac{p}{2}}$  is convex, we obtain
\begin{equation}\label{9}
\sum_{i=1}^{N}|X^{i}|^{p}\geq N^{1-\frac{p}{2}}\left(\sum_{i=1}^{N}|X^{i}|^{2}\right)^{\frac{p}{2}}= N^{1-\frac{p}{2}}| X|^{p}.
\end{equation}
Therefore, using the last two inequalities in the variational characterization of $\lambda_{1}(M)$, we get
\begin{equation}\label{10}
\lambda_{1}(M)\int_{\partial M}|X|^{p}d\mu_{h}\leq N^{\frac{p}{2}-1}n^{\frac{p}{2}}Vol_{\mu_{g}}(M).
\end{equation}
The end of the proof is the same that in case $1<p\leq2$.
\end{proof}
\begin{proof}[Proof of theorem \ref{t2}]
Similar to the proof of Theorem \ref{t1},  just enough to use  the generalized  Hsiung-Minkowski formula (\ref{eq1}) instead  of the classical one.
\end{proof}

\begin{proof}[Proof of  theorem \ref{t3}]
Similar to \cite{Ro4}, we assume that the function $t$ is a test function. Let $v=\langle \partial_{t},\nu\rangle=\langle \tilde\nabla t, \nu\rangle$. Hence, we have  $\Delta t=-nHv$ and
\begin{equation}\label{v1}
\int_{\Sigma}|\nabla t|^{2}d\mu_{g}=\int_{\Sigma}nHvt\,d\mu_{g}.
\end{equation}
Also, since $\nabla v=-S\nabla t$ we have
\begin{equation}\label{v2}
\int_{\Sigma}\langle S\nabla t,\nabla t\rangle d\mu_{g}=\int_{\Sigma}nHv^{2}\,d\mu_{g}.
\end{equation}
Then,
\begin{eqnarray}\nonumber
n\mathop {\inf } \limits_{\Sigma}( H)\int_{\Sigma}v^{2}d\mu_{g}&\leq& \int_{\Sigma}n H v^{2}d\mu_{g}\leq \int_{\Sigma}\langle S\nabla t,\nabla t\rangle d\mu_{g}\leq \kappa_{+}(\Sigma)\int_{\Sigma}|\nabla t|^{2}d\mu_{g}\\\label{v3}
&\leq&
 \kappa_{+}(\Sigma)\int_{\Sigma}nHvt\,d\mu_{g}\leq n \kappa_{+}(\Sigma)|H|_{\infty}\int_{\Sigma}vt\,d\mu_{g}\\\nonumber&\leq& n \kappa_{+}(\Sigma)|H|_{\infty}\left(\int_{\Sigma}|t|^{p}\,d\mu_{g}\right)^{\frac{1}{p}}\left(\int_{\Sigma}|v|^{\frac{p}{p-1}}\,d\mu_{g}\right)^{\frac{p-1}{p}}.
\end{eqnarray}
From the H\"{o}lder inequality we have
\begin{eqnarray*}
&&\mathop {\inf } \limits_{\Sigma}( H)\left( \int_{\Sigma}|v|^{\frac{p}{p-1}}d\mu_{g}\right)^{\frac{2(p-1)}{p}}Vol_{\mu_{g}}(\Sigma)^{\frac{2-p}{p}}\\&& \leq
\mathop {\inf } \limits_{\Sigma}( H)\int_{\Sigma}v^{2}d\mu_{g}\\&&\leq
 \kappa_{+}(\Sigma)|H|_{\infty}\left(\int_{\Sigma}|t|^{p}\,d\mu_{g}\right)^{\frac{1}{p}}\left(\int_{\Sigma}|v|^{\frac{p}{p-1}}\,d\mu_{g}\right)^{\frac{p-1}{p}},
\end{eqnarray*}
therefore,
\begin{equation}\label{v4}
\frac{\left(\int_{\Sigma}|v|^{\frac{p}{p-1}}\,d\mu_{g}\right)^{\frac{p-1}{p}}}{\left(\int_{\Sigma}|t|^{p}\,d\mu_{g}\right)^{\frac{1}{p}}}\leq \frac{
 \kappa_{+}(\Sigma)|H|_{\infty}}{\mathop {\inf } \limits_{\Sigma}( H)}Vol_{\mu_{g}}(\Sigma)^{\frac{p-2}{p}}.
\end{equation}
 On the other hand,  from the variational characterization of $\lambda_{1}(M)$, we have
\begin{equation}
\lambda_{1}(\Omega)\int_{\Sigma}|t|^{p}d\mu_{g}\leq \int_{\Omega}|\tilde{\nabla } t|^{p}d\mu_{\tilde{g}}.
\end{equation}
Since $|\tilde \nabla t|=1$ and $\tilde\Delta t=0$ we have
\begin{equation}
\int_{\Omega} |\tilde \nabla t|^{p}d\mu_{\tilde{g}}=Vol_{\mu_{\tilde{g}}}(\Omega)=\left(
\int_{\Omega} |\tilde \nabla t|^{2}d\mu_{\tilde{g}} \right)^{\frac{p}{2}}Vol_{\mu_{\tilde{g}}}(\Omega)^{1-\frac{p}{2}}
\end{equation}
and
\begin{equation}
\int_{\Omega} |\tilde \nabla t|^{2}d\mu_{\tilde{g}}=\int_{\Sigma}\langle t\tilde \nabla t,\nu\rangle d\mu_{g}=\int_{\Sigma} tv d\mu_{g}.
\end{equation}
By H\"{o}lder inequality we get
\begin{equation}
\int_{\Omega} |\tilde \nabla t|^{2}d\mu_{\tilde{g}}\leq \left(\int_{\Sigma}|t|^{p}\,d\mu_{g}\right)^{\frac{1}{p}}\left(\int_{\Sigma}|v|^{\frac{p}{p-1}}\,d\mu_{g}\right)^{\frac{p-1}{p}},
\end{equation}
thus, we obtain
\begin{equation}\label{v5}
\lambda_{1}(\Omega)\leq \frac{\left(\int_{\Sigma}|v|^{\frac{p}{p-1}}\,d\mu_{g}\right)^{\frac{p-1}{2}}}{\left(\int_{\Sigma}|t|^{p}\,d\mu_{g}\right)^{\frac{1}{2}}}Vol_{\mu_{g}}(\Omega)^{1-\frac{p}{2}}.
\end{equation}
Therefore, substituting  (\ref{v4})  in (\ref{v5}), we complete the proof of theorem.
\end{proof}

\begin{proof}[Proof of theorem \ref{tpq1}]
For coordinates functions $X^{k}$, by replacing if needed, $|X^{i}|^{r-2}X^{i}$ by $$|X^{i}|^{r-2}X^{i}-\frac{\int_{\partial M}|X^{i}|^{r-2}X^{i}d\mu_{h}}{Vol_{\mu_{h}}(\partial M)} $$
we can assume without loss  of generality,
\begin{equation*}
\int_{\partial M}|X^{i}|^{r-2}X^{i}d\mu_{h}=0
\end{equation*}
for all $i\in\{1,2,\cdots,N\}$. Thus, we can use the coordinates functions $X^{k}$ as test functions. \\
\begin{itemize}
\item[(1)]
The case $1<p<q<r\leq2$.\\
By the definition of $\sigma_{1}(M)$ we have
\begin{equation}\label{1pq}
\sigma_{1}(M)\int_{\partial M}\sum_{i=1}^{N}|X^{i}|^{r}dv_{h}\leq\int_{M}\sum_{i=1}^{N}\left(|\nabla X^{i}|^{p}+|\nabla X^{i}|^{q}\right)dv_{g}.
\end{equation}
Since $r\leq 2$,  we get $\left(\sum_{i=1}^{N}|X^{i}|^{2}\right)^{\frac{1}{2}}\leq \left(\sum_{i=1}^{N}|X^{i}|^{r}\right)^{\frac{1}{r}}$, then
\begin{equation}\label{2pq}
|X|^{r}=\left(\sum_{i=1}^{N}|X^{i}|^{2}\right)^{\frac{r}{2}}
\leq\sum_{i=1}^{N}| X^{i}|^{r}.
\end{equation}
On the other hand, the concavity of $y\to y^{\frac{p}{2}}$ yields
\begin{equation}\label{3pq}
\sum_{i=1}^{N}|\nabla X^{i}|^{p}=\sum_{i=1}^{N}\left(| \nabla X^{i}|^{2}\right)^{\frac{p}{2}}\leq N^{1-\frac{p}{2}}\left(\sum_{i=1}^{N}|\nabla X^{i}|^{2}\right)^{\frac{p}{2}}=N^{1-\frac{p}{2}} n^{\frac{p}{2}}.
\end{equation}
Similarly,
\begin{equation*}
\sum_{i=1}^{N}|\nabla X^{i}|^{q}\leq N^{1-\frac{q}{2}} n^{\frac{q}{2}},
\end{equation*}
Hence, we obtain
\begin{equation}\label{4pq}
\sigma_{1}(M)\int_{\partial M}|X|^{r}dv_{h}\leq \left( N^{1-\frac{p}{2}}n^{\frac{p}{2}}+N^{1-\frac{q}{2}}n^{\frac{q}{2}}\right)Vol(M),
\end{equation}
with multiply by $\left( \int_{\partial M}\left(|H|^{\frac{r}{r-1}}\right)dv_{h}\right)^{r-1}$ and use the integral H\"{o}lder inequality, we conclude that
\begin{equation}\label{5pq}
\sigma_{1}(M)\left|\int_{\partial M}\langle X, H\rangle dv_{h}\right|^{r}\leq \left( N^{1-\frac{p}{2}}n^{\frac{p}{2}}+N^{1-\frac{q}{2}}n^{\frac{q}{2}}\right)
\left( \int_{\partial M}\left(|H|^{\frac{r}{r-1}}\right)dv_{h}\right)^{r-1}
Vol(M).
\end{equation}
Now,  using the Hsiung-Minkouski formula
\begin{equation}\label{6pq}
\int_{\partial M}\left(\langle X, H\rangle+1 \right)dv_{h}=0
\end{equation}
we infer
\begin{equation}\label{7pq}
\sigma_{1}(M)\left(Vol(\partial M)\right)^{r}\leq  \left( N^{1-\frac{p}{2}}n^{\frac{p}{2}}+N^{1-\frac{q}{2}}n^{\frac{q}{2}}\right)
\left( \int_{\partial M}\left(|H|^{\frac{r}{r-1}}\right)dv_{h}\right)^{r-1}
Vol(M).
\end{equation}
\end{itemize}
\item[(2)] The case  $1<p<q\leq2$ and $r>2$.\\
It is straightforward that
\begin{equation}\label{8pq}
\sum_{i=1}^{N}\left( |\nabla X^{i}|^{p}+ |\nabla X^{i}|^{q}\right) \leq N^{1-\frac{p}{2}}n^{\frac{p}{2}}+N^{1-\frac{q}{2}}n^{\frac{q}{2}}.
\end{equation}
On the other hand, using the fact that $y\to y^{\frac{r}{2}}$  is convex, we obtain
\begin{equation}\label{9pq}
\sum_{i=1}^{N}|X^{i}|^{r}\geq N^{1-\frac{r}{2}}\left(\sum_{i=1}^{N}|X^{i}|^{2}\right)^{\frac{r}{2}}= N^{1-\frac{r}{2}}| X|^{r}.
\end{equation}
Therefore, using the last two inequalities in the variational characterization of $\sigma_{1}(M)$, we get
\begin{equation}\label{10pq}
\sigma_{1}(M)\int_{\partial M}|X|^{r}dv_{h}\leq N^{\frac{r}{2}-1}
\left( N^{1-\frac{p}{2}}n^{\frac{p}{2}}+N^{1-\frac{q}{2}}n^{\frac{q}{2}}\right)
Vol(M).
\end{equation}
\item[(3)] The case  $1<p\leq2$ and $2< q<r$.\\
It is straightforward that
\begin{equation}\label{80pq}
\sum_{i=1}^{N} |\nabla X^{i}|^{p} \leq N^{1-\frac{p}{2}}n^{\frac{p}{2}},
\end{equation}
and
\begin{equation}\label{81pq}
\sum_{i=1}^{N}|\nabla X^{i}|^{q} =\sum_{i=1}^{N}\left(| \nabla X^{i}|^{2}\right)^{\frac{q}{2}}\leq \left(\sum_{i=1}^{N}|\nabla X^{i}|^{2}\right)^{\frac{q}{2}}= n^{\frac{q}{2}}.
\end{equation}
On the other hand, using the fact that $y\to y^{\frac{r}{2}}$  is convex, we obtain
\begin{equation}\label{90pq}
\sum_{i=1}^{N}|X^{i}|^{r}\geq  N^{1-\frac{r}{2}}| X|^{r}.
\end{equation}
Therefore, using the last two inequalities in the variational characterization of $\sigma_{1}(M)$, we get
\begin{equation}\label{100pq}
\sigma_{1}(M)\int_{\partial M}|X|^{r}dv_{h}\leq N^{\frac{r}{2}-1}
\left( N^{1-\frac{p}{2}}n^{\frac{p}{2}}+n^{\frac{q}{2}}\right)
Vol(M).
\end{equation}
\item[(4)] The case  $2< p<q<r$.\\
It is straightforward that
\begin{equation}\label{82pq}
\sum_{i=1}^{N}(\left( |\nabla X^{i}|^{p}+ |\nabla X^{i}|^{q}\right) \leq  n^{\frac{p}{2}}+ n^{\frac{q}{2}}.
\end{equation}
On the other hand, using the fact that $y\to y^{\frac{r}{2}}$  is convex, we obtain
\begin{equation}\label{92pq}
\sum_{i=1}^{N}|X^{i}|^{r}\geq N^{1-\frac{r}{2}}| X|^{r}.
\end{equation}
Therefore, using the last two inequalities in the variational characterization of $\sigma_{1}(M)$, we get
\begin{equation}\label{101pq}
\sigma_{1}(M)\int_{\partial M}|X|^{r}dv_{h}\leq N^{\frac{r}{2}-1}(n^{\frac{p}{2}}+ n^{\frac{q}{2}})Vol(M).
\end{equation}
\end{proof}
\begin{proof}[Proof of theorem \ref{tpq1}]
Similar to the proof of Theorem \ref{tpq1}, just enough to use  the generalized  Hsiung-Minkowski formula (\ref{eq1}) instead  of the classical one.
\end{proof}

\begin{proof}[Proof  of the Theorem \ref{tpq4}]
Similar to \cite{Ro4}, we assume that the function $t$ is a test function. Let $v=\langle \partial_{t},\nu\rangle=\langle \tilde\nabla t, \nu\rangle$. Hence, we have  $\Delta t=-nHv$ and
\begin{equation}\label{v1pq}
\int_{\Sigma}|\nabla t|^{2}dv_{g}=\int_{\Sigma}nHvt\,dv_{g}.
\end{equation}
Also, since $\nabla v=-S\nabla t$ we have
\begin{equation}\label{v2pq}
\int_{\Sigma}\langle S\nabla t,\nabla t\rangle dv_{g}=\int_{\Sigma}nHv^{2}\,dv_{g}.
\end{equation}
Then,
\begin{eqnarray}\nonumber
n\mathop {\inf } \limits_{\Sigma}( H)\int_{\Sigma}v^{2}dv_{g}&\leq& \int_{\Sigma}n H v^{2}d\mu_{g}\leq \int_{\Sigma}\langle S\nabla t,\nabla t\rangle dv_{g}\leq \kappa_{+}(\Sigma)\int_{\Sigma}|\nabla t|^{2}dv_{g}\\\label{v3pq}
&\leq&
 \kappa_{+}(\Sigma)\int_{\Sigma}nHvt\,d\mu_{g}\leq n \kappa_{+}(\Sigma)|H|_{\infty}\int_{\Sigma}vt\,dv_{g}\\\nonumber&\leq& n \kappa_{+}(\Sigma)|H|_{\infty}\left(\int_{\Sigma}|t|^{r}\,dv_{g}\right)^{\frac{1}{r}}\left(\int_{\Sigma}|v|^{\frac{r}{r-1}}\,dv_{g}\right)^{\frac{r-1}{r}}.
\end{eqnarray}
From the H\"{o}lder inequality we have
\begin{eqnarray*}
&&\mathop {\inf } \limits_{\Sigma}( H)\left( \int_{\Sigma}|v|^{\frac{r}{r-1}}d\mu_{g}\right)^{\frac{2(r-1)}{r}}Vol_{v_{g}}(\Sigma)^{\frac{2-r}{r}}\\&& \leq
\mathop {\inf } \limits_{\Sigma}( H)\int_{\Sigma}v^{2}dv_{g}\\&&\leq
 \kappa_{+}(\Sigma)|H|_{\infty}\left(\int_{\Sigma}|t|^{r}\,dv_{g}\right)^{\frac{1}{r}}\left(\int_{\Sigma}|v|^{\frac{r}{r-1}}\,dv_{g}\right)^{\frac{r-1}{r}},
\end{eqnarray*}
therefore,
\begin{equation}\label{v4pq}
\frac{\left(\int_{\Sigma}|v|^{\frac{r}{r-1}}\,dv_{g}\right)^{\frac{r-1}{r}}}{\left(\int_{\Sigma}|t|^{r}\,dv_{g}\right)^{\frac{1}{r}}}\leq \frac{
 \kappa_{+}(\Sigma)|H|_{\infty}}{\mathop {\inf } \limits_{\Sigma}( H)}Vol(\Sigma)^{\frac{r-2}{r}}.
\end{equation}
 On the other hand,  from the variational characterization of $\sigma_{1}(M)$, we have
\begin{equation}
\sigma_{1}(\Omega)\int_{\Sigma}|t|^{r}dv_{g}\leq \int_{\Omega}\left(|\tilde{\nabla } t|^{p}+|\tilde{\nabla } t|^{q}\right)dv_{\tilde{g}}.
\end{equation}
Since $|\tilde \nabla t|=1$ and $\tilde\Delta t=0$ we have
\begin{equation}
\int_{\Omega} |\tilde \nabla t|^{p}dv_{\tilde{g}}=Vol(\Omega)=\left(
\int_{\Omega} |\tilde \nabla t|^{2}dv_{\tilde{g}} \right)^{\frac{r}{2}}Vol(\Omega)^{1-\frac{r}{2}}
\end{equation}
and
\begin{equation}
\int_{\Omega} |\tilde \nabla t|^{2}dv_{\tilde{g}}=\int_{\Sigma}\langle t\tilde \nabla t,\nu\rangle dv_{g}=\int_{\Sigma} tv dv_{g}.
\end{equation}
By H\"{o}lder inequality we get
\begin{equation}
\int_{\Omega} |\tilde \nabla t|^{2}dv_{\tilde{g}}\leq \left(\int_{\Sigma}|t|^{r}\,dv_{g}\right)^{\frac{1}{r}}\left(\int_{\Sigma}|v|^{\frac{r}{r-1}}\,dv_{g}\right)^{\frac{r-1}{r}},
\end{equation}
thus, we obtain
\begin{equation}\label{vp5}
\sigma_{1}(\Omega)\leq 2\frac{\left(\int_{\Sigma}|v|^{\frac{r}{r-1}}\,dv_{g}\right)^{\frac{r-1}{2}}}{\left(\int_{\Sigma}|t|^{r}\,dv_{g}\right)^{\frac{1}{2}}}Vol(\Omega)^{1-\frac{r}{2}}.
\end{equation}
Therefore, substituting  (\ref{v4pq})  in (\ref{vp5}), we complete the proof of theorem.
\end{proof}
{\bf Acknowledgements}\\
We thank the anonymous referees for their valuable comments, which improved the paper.\\
{\bf Funding}\\
 Not available.\\
{\bf Availability of data and materials}\\
Not applicable.\\
{\bf Competing interests}\\
The authors have declared that no competing interests exist.\\
{\bf Authors’ information}\\
Department of Pure Mathematics, Faculty of Science, Imam Khomeini International University, Qazvin, Iran. E-mails: azami@sci.ikiu.ac.ir.\\
{\bf Publisher’s Note}\\
Springer Nature remains neutral with regard to jurisdictional claims in published maps and institutional affiliations.

\end{document}